\documentclass{amsart}

\usepackage{amssymb}
\usepackage{amscd}
\usepackage[all,cmtip]{xy}
\usepackage{hyperref}
\usepackage{graphicx}

\usepackage[pdftex]{color}

\newcommand{\comment}[1]{}

\newcommand{\C}{\mathbb{C}}

\newcommand{\Ps}{\mathbb{P}}
\newcommand{\Q}{\mathbb{Q}}
\newcommand{\Z}{\mathbb{Z}}

\newcommand{\F}{\mathbb{F}}

\newcommand{\pa}{\mathfrak{p}}

\DeclareMathOperator{\disc}{disc}
\DeclareMathOperator{\Gal}{Gal}

\DeclareMathOperator{\Hom}{Hom}

\theoremstyle{plain}
\newtheorem{theorem}{Theorem}[section]

\newtheorem{corollary}[theorem]{Corollary}

\newtheorem{lemma}[theorem]{Lemma}

\newtheorem{proposition}[theorem]{Proposition}

\theoremstyle{definition}
\newtheorem{definition}[theorem]{Definition}

\newtheorem{remark}[theorem]{Remark}

\newtheorem{example}[theorem]{Example}

\begin{document}

\title[Stability of genus]{Genus growth in $\Z_p$-towers of function fields}

\author{Michiel Kosters}
\address{University of California, Irvine, Department of
Mathematics, 340 Rowland Hall, Irvine, CA 92697}
\email{kosters@gmail.com}
\author{Daqing Wan}
\address{University of California, Irvine, Department of
Mathematics, 340 Rowland Hall, Irvine, CA 92697}
\email{dwan@math.uci.edu}

\date{\today}

\maketitle

\begin{abstract}
Let $K$ be a function field over a finite field $k$ of characteristic $p$ and let $K_{\infty}/K$ be a geometric extension with Galois group $\Z_p$. Let $K_n$ be the corresponding subextension with Galois group $\Z/p^n\Z$ and genus $g_n$. In this paper, we give a simple explicit formula  $g_n$ in terms of an explicit 
Witt vector construction of the $\Z_p$-tower. This formula leads to a tight lower bound on $g_n$ which is quadratic in $p^n$. Furthermore, we determine  
all $\Z_p$-towers for which the genus sequence is stable, in the sense that there are $a,b,c \in \Q$ such that $g_n=a p^{2n}+b p^n +c$ for $n$ large enough. 
Such genus stable towers are expected to have strong stable arithmetic properties for their zeta functions. 
A key technical contribution 
of this work is a new simplified formula for the Schmid-Witt symbol coming from local class field theory.
\end{abstract}

\maketitle

\section{Introduction}

\subsection{Global function fields}

Let $K$ be a function field over a finite field $k$ of characteristic $p$. Let 
$$K=K_0 \subset K_1 \subset K_2 \subset \cdots \subset K_{\infty}$$
be a geometric $\Z_p$-tower of function fields such that ${\rm Gal}(K_n/K) = \Z_p/p^n\Z_p$. Let $g_n$ denote the genus 
of $K_n$. We assume that the tower is only ramified at a finite number of places of $K$.  In the spirit of Iwasawa theory, an emerging new research area is to study 
the possible stable arithmetic properties for the sequence of zeta functions of $K_n$ as $n$ varies, see \cite{WAN7} and \cite{KO9} for recent progresses 
and the relevant references there.  A necessary condition for the sequence of zeta functions to be arithmetically stable is that the genus sequence 
$g_n$ must be stable in the sense that $g_n$ is a quadratic polynomial in $p^n$ for large $n$. 
The aim of this paper is to classify all genus stable $\Z_p$-towers of $K$.  

First, we give an explicit construction of all geometric $\Z_p$-towers of $K$ using Witt vectors of $K$, via an improved presentation of 
the classical Artin-Schreier-Witt theory.  This explicit construction leads to a simple explicit genus formula for the genus sequence, see Theorem \ref{831}. 
As an application, we derive an explicit quadratic  lower bound in $p^n$ for $g_n$,  which is tight in many cases. This explicit formula also allows us to derive 
a simple criterion for when the genus $g_n$ is a quadratic polynomial in $p^n$ for large $n$.

By the Riemann-Hurwitz formula, the  genus can be calculated from local ramification information and we can reduce to the local case where there is only 
one ramified prime. To illustrate our result in this introduction, here we 
consider an essential case, that is, $K=k(X)$ is the rational function field and the geometric $\Z_p$-tower $K_{\infty}$  is only ramified at infinity. 
Then any such $\Z_p$-tower $K_{\infty}$ over $K=k(X)$ 
can be uniquely constructed from a constant $c\in \Z_p$ and 
a primitive convergent power series 
$$f(X) = \sum_{(i,p)=1} c_i X^i \in \Z_q[[X]], \ c_i \in \Z_q, \  \lim_i c_i =0,$$ 
where $\Z_q =W(\F_q)$ denotes the Witt vectors of $\F_q$, and $f(X)$ is called primitive if not all $c_i$ are divisible by $p$.  
The construction is explicitly given by the following Witt vector equation 
$$K_{\infty}: (Y_0^p, Y_1^p, \cdots) - (Y_0, Y_1, \cdots) = c\beta + \sum_{(i,p)=1}c_i (X^i, 0, \cdots),$$
where both sides are Witt vectors and $\beta$ is any fixed element of $\Z_q$ with trace $1$. 

\begin{theorem} Let $K_{\infty}$ be a geometric $\Z_p$-tower ramified only at infinity as constructed above by a primitive 
convergent power series $f(X)$.  Then, we have 

(1). For each integer $n\geq 1$, the genus $g_n$ is given by the following formula 
$$2g_n = \sum_{k=1}^n (p-1)p^{k-1} \left( -1 + \max_{i:\ v_p(c_i)<k} \left\{ ip^{k-1 -v_p(c_i)}\right\}\right),$$
where $v_p$ denotes the standard $p$-adic valuation with $v_p(p)=1$.  

(2). For any $\epsilon >0$, there is a constant $c(\epsilon)$ such that for all $n>c(\epsilon)$,  we have 
$$g_n \geq \frac{p^{2n}}{2(p+1)+\epsilon}.$$ 

(3). The tower $K_{\infty}$ is genus stable in the sense that for all large enough $n$ one has
$$g_n=a p^{2n}+b p^n+c, \ a,b,c \in \Q$$ 
if and only if 
$$d: = \max_{(i,p)=1} \left\{ \frac{i}{p^{v_p(c_i)}}\right\}$$ exits (and is a thus finite rational number).

\end{theorem}

Remarks. Part (2) shows that the genus sequence $g_n$ grows at least quadratically in $p^n$. The lower bound in (2) cannot be improved in general.  
In particular, it implies that the lower bound for the genus in the literature is incorrect: the $\epsilon$ cannot be dropped (Remark \ref{goa}). 
The proof of the above theorem follows from Corollary \ref{15b}, Proposition \ref{blaat} and  Proposition \ref{hulu2}.


%

\subsection{Local function fields}

Set $K=k((T))$, where $k$ is a finite field of characteristic $p$ and cardinality $q$. Local class field theory studies the abelian Galois extensions of $K$. Combining local class field theory and the theory of Artin-Schreier-Witt extensions gives us the so-called Schmid-Witt symbol
\begin{align*}
[\ ,\ )_n: W_n(K) \times K^* \to \Z_p/p^n\Z_p, 
\end{align*}
where $W_n(K)$ is the ring of Witt-vectors of $K$ of length $n$, and $1\leq n \leq \infty$. The strongest case is when $n=\infty$, in which 
case the symbol $[\ ,\ )_{\infty}$ will be simply denoted by $[\ ,\ )$. 
In the simplest classical case $n=1$,  the symbol $[\ ,\ )_1$, coming from Artin-Schreier theory, has the following beautiful simple formula: 
\begin{align*}
[\ ,\ )_1: K \times K^* \to& \Z/p\Z \\
(x,y) \mapsto& \mathrm{Tr}_{k/\F_p} \left( \mathrm{Res}(x dy/y) \right)
\end{align*}
where $dy$ is the derivative of $y$ (see \cite{SE1}). 

As $n$ grows,  the situation becomes more complicated. Various formulas for $[\ ,\ )_n$ for finite $n$ were essentially known, but they do not resemble the simple 
formula for $[\ ,\ )_1$.  In fact, they are all 
quite complicated (see for example \cite{THO}) involving the ghost coordinates of Witt vectors.  
We have found a new and simple formula for $[\ ,\ )$ (and thus for $[\ ,\ )_n$ for all $n$) which resembles the one for $[\ ,\ )_1$ without using ghost coordinates. 
In the formula below, $\tilde{x}$ and $\tilde{y}$ are very explicit (see Theorem \ref{100}).

\begin{theorem}
\begin{align*}
[x,y) = \mathrm{Tr}_{\Z_q/\Z_p} \left( \mathrm{Res}( \tilde{x} \cdot d\tilde{y}/\tilde{y} \right) ).
\end{align*}
\end{theorem}

The simple nature of the above formula allows for easy computation of conductors and higher ramification groups of all $\Z_p$-towers, as in Proposition \ref{715}. 
These are the key technical results for our genus calculations, which might be of independent interests.

\begin{remark}
Many proofs in this paper, mostly regarding Artin-Schreier-Witt theory and Schmid-Witt symbols, have been removed since these results are mostly known or can be derived easily from known results. For an extended version of this paper with complete proofs,  see \cite{WAN8}. 
\end{remark}

\section{Artin-Schreier-Witt extensions}

\subsection{Witt vectors}

For a detailed description, see \cite{THO}, \cite{RABI}, or follow the exercises from \cite[Chapter VI, Exercises 46-48]{LA}. We will give a brief summary which we will use as a black box. 

Let $p$ be a prime number. Let $R$ be a commutative ring with identity. We define the ring of $p$-typical Witt vectors $W(R)=W_{\infty}(R)$ as follows. 

\begin{definition}
Let $\mathcal{C}$ be the category of commutative rings with identity. Then there is a unique functor $W: \mathcal{C} \to \mathcal{C}$ such that the following hold:
\begin{itemize}
\item For a commutative ring $R$, one has $W(R)=R^{\Z_{\geq 0}}$ as sets. \\
\item If $f: R \to S$ is a ring morphism, then the induced ring morphism satisfies $W(f) ((r_i)_i)=(f(r_i))_i$. \\
\item The map $g=(g^{(i)})_i: W(R) \to R^{\Z_{\geq 0}}$ defined by
\begin{align*}
(r_i)_i \to \left( \sum_{j=0}^i p^j r_j^{p^{i-j}} \right)_i.
\end{align*}
is a ring morphism (where $R^{\Z_{\geq 0}}$ has the product ring structure).
\end{itemize}
\end{definition}
 In $W(R)$ one has
\begin{align*}
(r_0,\ldots) + (r_0',\ldots) = (r_0+r_0', \ldots), \\
(r_0,\ldots) \cdot (r_0',\ldots) = (r_0 \cdot r_0', \ldots),
\end{align*}
where the formulas for the later coordinates are quite complicated.
The zero element of $W(R)$ is $(0,0,\ldots)$ and the identity element is $(1,0,0,\ldots)$. 
One has $W(\F_p)=\Z_p$. If $k$ is a finite field of $q$ elements, then $W(k)$ is isomorphic to the ring $\Z_q$ of integers of the unramified field extension of $\Q_p$ with residue field $k$. 

The above map $g$ is called the ghost map, and this map is an injection if $p$ is not a zero divisor in $R$. This ghost map, together with functoriality, determines the ring structure.
Furthermore, we have the Teichm\"uller map
\begin{align*}
[\ ]: R \to& W(R) \\
r \mapsto& (r,0,0,\ldots),
\end{align*}
This map is multiplicative: for $r,s \in R$ one has $[rs]=[r][s]$. 
We have the so-called Verschiebung group morphism
\begin{align*}
V: W(R) \to& W(R) \\
(r_0,r_1,r_2,\ldots) \mapsto& (0,r_0,r_1,r_2,\ldots). 
\end{align*}
We make $W(R)$ into a topological ring as follows. The open sets around $0$ are the sets of the form $V^i W(R)$. We call this the $V$-adic topology.  With this topology, $W(R)$ is complete and Hausdorff. Furthermore, a ring morphism $R \to S$ induces a continuous map $W(R) \to W(S)$. Any $r=(r_0,r_1,\ldots) \in W(R)$ can be written as $r=\sum_{i=0}^{\infty} V^i[r_i]$.

Now let us restrict to the case where $R=K$ is a field of characteristic $p$. The ring $W(K)$ has the subring $W(\F_p) = \Z_p$. Witt vectors $(x_0,x_1,\ldots) \in W(K)$ with $x_0 \neq 0$, have a multiplicative inverse (note that $W(K)$ is not a field, since $p$ is not invertible). The Frobenius map $x \mapsto x^p$ on $K$ induces a ring morphism 
\begin{align*}
F: W(K) \to& W(K) \\
(r_0,r_1,\ldots) \mapsto& (r_0^p,r_1^p,\ldots).
\end{align*}
 In fact, one has $V F = FV= \cdot p$. One also sees that $W(K)$ is a torsion-free $\Z_p$-module. Let $K'/K$ be a Galois extension and let $g \in G=\Gal(K'/K)$. Then we have a map $W(g): W(K') \to W(K')$. If $K'/K$ is finite Galois, we define the following $W(K)$-linear trace map
\begin{align*}
\mathrm{Tr}_{W(K')/W(K)}: W(K') \to& W(K) \\
x \mapsto& \sum_{g \in G} W(g)x. 
\end{align*}

\subsection{Artin-Schreier-Witt theory}

For a full treatment of Artin-Schreier-Witt theory, see \cite{WAN8}. 

Fix a prime $p$ and let $K$ be a field of characteristic $p$. Let $K^{\mathrm{sep}}$ be a separable closure of $K$. We define a group morphism 
\begin{align*}
\wp=F-\mathrm{id}: W(K^{\mathrm{sep}}) \to& W(K^{\mathrm{sep}}) \\
x \mapsto& Fx-x,
\end{align*}
with kernel $\Z_p$. One can easily show that this map is surjective.  For $a \in W(K)$ and $x=(x_0,x_1, \ldots) \in \wp^{-1}a \subset W(K^{\mathrm{sep}})$,  we set $K(\wp^{-1} a)=K(x_0,x_1,\ldots)$. This extension does not depend on the choice of $x$. In fact, $K(\wp^{-1}a)=K(\wp^{-1}b)$ if $a \equiv b \pmod{ \wp W(K)}$.

We endow $W(K)/\wp W(K)$ with the induced topology, which is the same as the $p$-adic topology (where a basis of open sets around zero is given by $\{ p^i (W(K)/\wp W(K)): i \in \Z_{\geq 0}\}$). We endow Galois groups with the Krull topology. 
The main theorem we need from Artin-Schreier-Witt theory is the following.

\begin{theorem} \label{oko}
Let $K$ be a field of characteristic $p$ with absolute abelian Galois group $G=\Gal(K^{\mathrm{ab}}/K)$ and with $p$-part $G_p$. Then as topological groups one has
\begin{align*}
W(K) / \wp W(K) \cong& \Hom_{\mathrm{cont}}(G_p, \Z_p) \\
\overline{a}=\overline{\wp x} \mapsto& \left( g \mapsto gx-x \right). 
\end{align*}
Furthermore, the field extension of $K$ corresponding to $\overline{a}=\overline{\wp x} \in W(K) / \wp W(K)$ under this bijection and Galois theory is equal to $K(\wp^{-1}a)$ and the map 
\begin{align*}
\Gal(K(\wp^{-1}a)/K) \to& \Hom(\Z_p \overline{a} , \Z_p) \\
g \mapsto (r \overline{a} \mapsto& r (gx-x))
\end{align*}
 is an isomorphism of topological groups.
\end{theorem}
\begin{proof}
See  \cite[Theorem 3.4, Theorem 3.6]{WAN8}
\end{proof}

Let $H$ be an abelian group. We define its $p$-adic completion, a $\Z_p$-module, to be 
\begin{align*}
\widehat{H}=\underset{\underset{i}{\leftarrow}}\lim\ H/p^iH.
\end{align*}
We make $\widehat{H}$ into a topological group by giving a basis $\left\{ p^i \widehat{H}: i \in \Z_{\geq 0}\right\}$ around $0$. We call this the $p$-adic topology. 

For $x \in K$,  we set $\wp x = x^p-x$. One easy proposition is the following.

\begin{lemma} \label{182}
Let $K$ be a field of characteristic $p$. Let $\mathfrak{B}$ be a basis of $K/\wp K$ over $\F_p$. Then the map
\begin{align*}
\widehat{ \bigoplus_{\mathfrak{B}} \Z_p} \to& W(K)/\wp W(K) \\
(a_b)_{b \in \mathfrak{B}} \to& \sum_{i} a_b [b] \pmod{\wp W(K)}
\end{align*}
is an isomorphism of topological groups. 
\end{lemma}
\begin{proof}
See  \cite[Proposition 3.10]{WAN8}. 
\end{proof}

\begin{example} \label{183}
In certain cases, one can easily find a basis of $K/\wp K$ over $\F_p$. Below we will construct a subset $\mathcal{D}$ of $K$ which injects into $K/\wp K$ and such that its image forms an $\F_p$-basis of $K/\wp K$. 
For this purpose, it is enough to show that $\mathrm{Span}_{\F_p}(\mathcal{D}) \cap \wp K = 0$ and $\mathrm{Span}_{\F_p}(\mathcal{D}) + \wp K=K$. 
\begin{itemize}
\item Assume $K$ is a finite field. Take any vector $b$ with $b \not \in \wp K$, that is, take any $b \in K$ with $\mathrm{Tr}_{K/\F_p}(b) \neq 0$. One can take $\mathcal{D}=\{b\}$.
\item Assume $K=k((T))$ for a perfect  field $k$. Let $\mathfrak{B}$ be a subset of $k$ giving a basis of $k/\wp k$ over $\F_p$ and let $\mathfrak{C}$ be a basis of $k$ over $\F_p$. Then one can take
\begin{align*}
\mathcal{D}= \mathfrak{B} \sqcup \{ c T^{-i}: c \in \mathfrak{C},\ (i,p)=1,\  i\geq 1\}.
\end{align*}
Let us prove this result. If $f \in Tk[[T]]$, then set $g=- \sum_{i=0}^{\infty} f^{p^i} \in Tk[[T]]$. One has $\wp g =f$. Note that $a_iT^{-ip} \equiv a_i^{1/p} T^{-i} \pmod{ \wp K}$ (where we use that $k$ is perfect). Hence we find $\mathrm{Span}_{\F_p}(\mathcal{D}) + \wp K=K$.  Let $f = \sum_i a_i T^i$. One has $\wp f=  \sum_i a_i^p T^{ip}- \sum_i a_i T^i$. We find $\mathrm{Span}_{\F_p}(\mathcal{D}) \cap \wp K = 0$.
\item Assume $K=k(X)$ for some perfect field $k$. Let $\mathfrak{B}$ be a subset of $k$ giving a basis of $k/\wp k$ over $\F_p$ and let $\mathfrak{C}$ be a basis of $k$ over $\F_p$. Then one can take
\begin{align*}
\mathcal{D}= \mathfrak{B} \sqcup& \bigsqcup_f \left\{ \frac{bX^i}{f^j}:\ (j,p)=1,\ j\geq 1, \ 0 \leq i< \deg(f),\ b \in \mathcal{C} \right\} \\ \sqcup& \{ b X^j,\ (j,p)=1,\ j\geq 1, \ b \in \mathcal{C}\},
\end{align*}
where $f\in k[X]$ runs over monic irreducible polynomials. One can easily show this by using partial fractions. 
\end{itemize}
\end{example}

\section{Local function fields } \label{s3}

Let $k$ be a finite field of cardinality $q$ and characteristic $p$. We set $\Z_q=W(k)$. Let $K=k((T))$. The field $K$ has a natural valuation. If $f=\sum_{i \geq v} a_i T^i$ with $a_v \neq 0$, then the valuation is $v$. We set $\pa=Tk[[T]]$, the unique maximal ideal of $k[[T]]$.

Let $K^{\mathrm{ab}}$ be the maximal abelian extension of $K$. Let $G=\Gal(K^{ab}/K)$ with $p$-part $G_p$, all endowed with the Krull topology. Set $\widehat{K^*}= \underset{\underset{n}{\leftarrow}}{\lim}\ K^*/(K^*)^{p^n}$, the $p$-adic completion of $K^*$ with its natural $p$-adic topology. Note that $\widehat{K^*} \cong T^{\Z_p} \times U_1$ where $U_1=1+\pa$ are the one units of $K$. We usually identify $\widehat{K^*}$ with $T^{\Z_p} \times U_1$. We have a natural map $K^* \to \widehat{K^*}$, with kernel $k^*$. 

The Artin map (or Artin reciprocity law) from class field theory is a certain group morphism $K^* \to G$ (see \cite{SE1}). This map is usually the best way to understand the group $G$ and to understand ramification in abelian extensions of $K$.
This Artin map induces a homeomorphism
\begin{align*}
\psi: \widehat{K^*} \to G_p.
\end{align*}
Theorem \ref{oko} gives an isomorphism $W(K) / \wp W(K) \to \Hom_{\mathrm{cont}}(G_p, \Z_p)$. If we combine both maps, we obtain a $\Z_p$-bilinear, hence continuous, symbol
\begin{align*}
[\ ,\ ): W(K)/\wp W(K) \times \widehat{K^*} \to  \Z_p \\
(\overline{\wp x},y) \mapsto& \psi(y)x-x.
\end{align*}
This symbol is often called the Schmid-Witt symbol. For $1\leq n \leq \infty$, reducing module $p^n$ gives the level $n$ Schmid-Witt 
symbol 
$$[\ ,\ )_n: W_n(K)/\wp W_n(K) \times \widehat{K^*} \to  \Z_p/p^n\Z_p$$
mentioned in the introduction, where $W_n(K)$ denotes the length $n$ Witt vectors. 
 
Note that the group $W(K)/\wp W(K)$ can be described as follows. 

\begin{proposition} \label{712}
Let $\alpha \in k$ with $\mathrm{Tr}_{k/\F_p}(\alpha) \neq 0$ and set $\beta=[\alpha] \in \Z_q \subset W(K)$. Then any $x \in W(K)$ has a unique representative in $W(K)/\wp W(K)$ of the form 
\begin{align*}
c \beta + \sum_{i \geq 1, (i,p)=1} c_i [ T^{-i} ].
\end{align*}
 with $c \in \Z_p $ and $c_i \in \Z_q$ with $c_i \to 0$ as $i \to \infty$.
\end{proposition}
\begin{proof}
This follows from Lemma \ref{182} and Example \ref{183}. 
\end{proof}

Combining known formulas for $[\ ,\ )$ as in \cite{THO} together with our insight of Proposition \ref{712} allows one to prove a simple formula for $[\ ,\ )$, 
which we now describe. 

Consider the ring 
\begin{align*}
R=\left\{\sum_{i \in \Z} a_i T^i: a_i \in \Z_q,\ \lim_{i \to -\infty} a_i=0\right\}= \underset{\underset{i}{\leftarrow}}\lim\ \Z_q/p^i \Z_q ((T))
\end{align*}
of two sided power series with some convergence property. We have a residue map
\begin{align*}
\mathrm{Res}: R \to& ~\Z_q \\
\sum_i a_i T^i \to& ~a_{-1}. 
\end{align*}
Let $x \in W(K)$. Let $c\beta + \sum_{(i,p)=1} c_i [  T]^{-i}$ be its unique representative modulo $\wp W(K)$ as in Proposition \ref{712}. We define a map 
\begin{align*}
\tilde{}: W(K)/\wp W(K) \to& R \\
c\beta + \sum_{(i,p)=1} c_i [  T]^{-i} \pmod{\wp W(K)} \mapsto& c\beta + \sum_{(i,p)=1} c_i T^{-i}.
\end{align*}
Any element $y \in \widehat{K^*} \cong T^{\Z_p} \times \left( 1+ Tk[[T]] \right)$ can uniquely be written as (with some abuse of notation)
\begin{align*}
y=T^e \cdot \prod_{(i,p)=1}^{\infty}  \prod_{j=0}^{\infty} (1- a_{ij} T^i)^{p^j}
\end{align*}
with $e \in \Z_p$ and $a_{ij} \in k$. 
We define another map 
\begin{align*}
\tilde{ }: \widehat{K^*} \cong T^{\Z_p} \times \left( 1+ Tk[[T]] \right) \to& T^{\Z_p} \times \left( 1+ T\Z_q[[T]] \right) \\
T^e \cdot \prod_{(i,p)=1}^{\infty}  \prod_{j=0}^{\infty}  (1- a_{ij} T^i)^{p^j} \mapsto& T^e \cdot \prod_{(i,p)=1}^{\infty}  \prod_{j=0}^{\infty}  (1- [a_{ij}]T^i)^{p^j}.
\end{align*}
Furthermore, we define the group morphism
\begin{align*}
\mathrm{dlog}: T^{\Z_p} \times \left( 1+ T\Z_q[[T]] \right) \to& \Z_q((T)) \\
T^e \cdot f \mapsto& \frac{e}{T} + \frac{df}{f}
\end{align*}
where $df$ is the formal derivative of $f$. We have the following formula for $[\ ,\ )$, which resembles formulas for the simpler symbol $[\ ,\ )_1$ as in \cite{SE1}.

\begin{theorem} \label{100}
Let $x \in W(K)$ and $y \in \widehat{K^*}$. Then one has 
\begin{align*}
[x,y) = \mathrm{Tr}_{\Z_q/\Z_p} \left( \mathrm{Res}( \tilde{x} \cdot \mathrm{dlog}\tilde{y}) \right).
\end{align*}
Equivalently, let $x \equiv c\beta + \sum_{(i,p)=1} c_i [ T]^{-i} \pmod{\wp W(K)}$ as in Proposition \ref{712}, and $y=T^e \cdot \prod_{(i,p)=1}^{\infty}  \prod_{j=0}^{\infty}  (1- a_{ij} T^i)^{p^j} \in \widehat{K^*}$ with $a_{ij} \in k$ and $e \in \Z_p$. Then one has:
\begin{align*}
[x,y) = ce\mathrm{Tr}_{\Z_q/\Z_p}(\beta)- \sum_{j=0}^{\infty} p^j  \mathrm{Tr}_{\Z_q/\Z_p} \left(\sum_{(i,p)=1} c_i \sum_{l|i} l [a_{lj}]^{i/l} \right) \in \Z_p. 
\end{align*}
\end{theorem}
\begin{proof}
See \cite[Theorem 4.7]{WAN8}. We deduce the formula from the formulas in \cite{THO}. 
\end{proof}

 For a finite abelian extension $L/K$,  the Artin map induces a map $\psi_L: K^* \to \Gal(L/K)$. The conductor of $L/K$ is defined to be $\mathfrak{f}(L/K)=\pa^i$ where $i$ is minimal such that $U_i \subseteq \ker(\psi_L)$, where $U_i=1+\pa^i$ for $i \geq 1$ and $U_0=k[[T]]^*$. The explicit formula above allows one to easily compute conductors. 

\begin{proposition} \label{715}
Let $x=\wp (y_0,y_1,\ldots) \in W(K)$ with $x \equiv c \beta + \sum_{(i,p)=1} c_i [ T]^{-i} \pmod{\wp W(K)}$ as in Proposition \ref{712}.
Let $n \in \Z_{\geq 1}$. One has 
\begin{align*}
\mathfrak{f}_n:=\mathfrak{f}(K(y_0,y_1,\ldots,y_{n-1})/K)=\pa^{u_n}
\end{align*}
with
\begin{align*}
u_n = \left\{  \begin{array}{cc} 1+ \max\{ ip^{n-v(c_i)-1} : i \textrm{ such that } v(c_i)<n\} & \mathrm{if\ }\exists i: v(c_{i})<n \\
0 & \mathrm{otherwise}. \end{array} \right.
\end{align*}
\end{proposition}
\begin{proof}
 See \cite[Proposition 4.14]{WAN8}.
\end{proof}

\begin{remark}
Let us give an essentially equivalent version of Proposition \ref{715} in terms of upper ramification groups. Let $x \in W(K)$ as in Proposition \ref{715}. For $r \in \Z_{\geq 0}$ consider the $r$-th upper ramification group 
\begin{align*}
H^r=\psi_{K(\wp^{-1}x)}(U_r) \subseteq \Gal(K(\wp^{-1}x)/K) \cong  \Hom(\Z_p \overline{x} , \Z_p). 
\end{align*}
One then has
\begin{align*}
H^r=\left\{ \tau: \tau \overline{x} \in p^{b_r} \Z_p \right\} \subseteq \Hom(\Z_p \overline{x}, \Z_p),
\end{align*}
where
\begin{align*}
b_r = \left\{ \begin{array}{cc} 
\min\{ v(c_i)+ \lceil  \log_p \left(\frac{r}{i}\right) \rceil: (i,p)=1 \}  & \textrm{if } r \geq 1 \\
\min\{ v(c_i) : (i,p)=1 \} & \textrm{if } r=0.
\end{array} \right.
\end{align*}
See  \cite[Proposition 4.14]{WAN8} for the details.
\end{remark}

\section{Global function fields}

Let $k$ be a finite field of characteristic $p$. Let $K=K_0$ be a function field over $k$ (a finitely generated field extension of $k$ of transcendence degree $1$) with full constant field $k$. Let $x=(x_0,x_1,\ldots)=\wp (y_0,y_1,\ldots) \in W(K)$. This Witt vector defines a field extension $K_{\infty}/K$. For simplicity, we assume that $x_0 \not \in \wp K$. Set $K_i=K(y_0,y_1,\ldots,y_{i-1})$. One then has a tower of fields $K=K_0 \subset K_1 \subset K_2 \subset \ldots \subset K_{\infty}=K(y_0,y_1,\ldots)$ with $\Gal(K_n/K) \cong \Z/p^n\Z$ and $\Gal(K_{\infty}/K) \cong \Z_p$.  

\subsection{Genus formula}

Let $\pa$ be a place of $K$ with residue field $k_{\pa}$ and uniformizer $\pi_{\pa}$. Then, locally, this extension is given by $x=(x_0,x_1,\ldots) \in W(K_{\pa})$ where $K_{\pa} \cong k_{\pa}((\pi_{\pa}))$ is the completion at $\pa$ (by the Cohen structure theorem). Let $\alpha_{\pa} \in k_{\pa}$ with $\mathrm{Tr}_{k_{\pa}/\F_p}(\alpha_{\pa}) \neq 0$. Set $\beta_{\pa}=[\alpha_{\pa}] \in W(k_{\pa})$. One has 
$$x \equiv c_{\pa} \beta_{\pa} + \sum_{(i,p)=1} c_{\pa,i} [\pi_{\pa}]^{-i} \pmod{\wp W(K_{\pa})}$$
with $c_{\pa} \in \Z_p$ and $c_{\pa,i} \in W(k_{\pa})$ and $c_{\pa,i} \to 0$ as $i \to \infty$ (Proposition \ref{712}). 
Proposition \ref{715} then shows that the conductor at $\pa$ of $K_n/K$ is equal to $\mathfrak{f}_{\pa,n} =\pa^{u_{\pa,n}}$ with
\begin{align*}
u_{\pa,n} = \left\{  \begin{array}{cc} 1+ \max\{ ip^{n-v(c_{\pa,i})-1}  : i \textrm{ such that } v(c_{\pa,i})<n\} & \mathrm{if\ }\exists i: v(c_{\pa,i})<n \\
0 & \mathrm{otherwise}. \end{array} \right.
\end{align*}
The conductor of $K_n/K$ is the formal expression
\begin{align*}
\mathfrak{f}_n=\mathfrak{f}(K_n/K) = \prod_{\pa} \mathfrak{f}_{\pa,n},
\end{align*}
which is a finite product. 

Let $g_n$ be the genus of $K_n$, where the genus is the genus of the corresponding smooth projective curve defined by $K_n$ over the the integral closure of $k$ inside $K_n$.  
We let $n_c$ be maximal such that $K_{n_c}/K$ is a constant field extension.
\begin{theorem} \label{831}
For $n \in \Z_{\geq 1}$,  we have 
\begin{align*}
p^{\min\{n_c,n \}}(2 g_n - 2) = p^{n} ( 2 g_0 - 2 ) + \sum_{\pa} [k_{\pa}:k] \sum_{i=1}^n \varphi(p^i) u_{\pa,i}.
\end{align*}
\end{theorem}
\begin{proof}
This is an application of the Riemann-Hurwitz formula, together with the F\"uhrerdiskriminantenproduktformel. See \cite[Theorem 5.2]{WAN8}. 
\end{proof}

\subsection{Genus lower bound}

We assume that the tower is not a constant extension, otherwise, the genus $g_n$ will be a constant. 
This means that $n_c$ is finite. Since $K_{\infty}/K$ is abelian and infinite, 
by class field theory, the extension $K_{\infty}/K$ must be 
ramified at some prime.  Let $n_u$ be maximal such that $K_{n_u}/K$ is unramified. Then $n_c \leq n_u <\infty$. 

\begin{corollary} \label{15b}
Let $n \geq n_u$. The following statements hold:

\begin{enumerate}
	\item
\begin{align*}
p^{n_c} (2g_n-2)  \geq p^n (2g_0-2) +p^n-p^{n_u}+ p^{n_u} \frac{p^{2(n-n_u)}-1}{p+1} .
\end{align*}

\item   
\begin{align*}
\liminf_{n \to \infty} \frac{g_n}{p^{2n}} \geq \frac{1}{2p^{n_u+n_c}(p+1)}.
\end{align*}

\item  For any $\epsilon>0$,  there is an integer $m$ such that for all $n \geq m$ one has
\begin{align*}
g_n \geq \frac{p^{2n-n_u-n_c}}{2(p+1)+\epsilon} \geq \frac{p^{2n-n_u-n_c-1}}{3+\epsilon}.
\end{align*}
\end{enumerate}
\end{corollary}
\begin{proof}
We try to make the genus as small as possible in the genus formula. The smallest genus is obtained if only one prime $\pa$ is ramified with $[k_{\pa}:k]=1$, such that $u_{\pa,n}=1+p^{n-n_u-1}$ for $n>n_u$, and $u_{\pa,n}=0$ for $n \leq n_u$ (see Proposition \ref{715}).
One finds for $n \geq n_u$ by Theorem \ref{831}:
\begin{align*}
p^{n_c} (2g_n-2)  \geq&  p^n (2g_0-2) + \sum_{i=n_u+1}^n \varphi(p^i) \left( 1 + p^{i-n_u-1} \right) \\
=& p^n (2g_0-2) +p^n-p^{n_u}+ p^{n_u} \frac{p^{2(n-n_u)}-1}{p+1}.  
\end{align*}
The first part is proved.
The second and third part follow by looking at the last term of the formula from the first part.
\end{proof}

\begin{remark} \label{goa}
The bounds in Corollary \ref{15b} are often sharp when the $p$-part of the class group of $K$ is $0$. In particular, the bounds are sharp when $K=k(X)$, the projective line. We will give explicit examples later.

Gold and Kisilevsky in  \cite[Theorem 1]{GOL} state that for large $n$, if $n_c=0$ one has:
\begin{align*}
g_n \geq \frac{p^{2(n-n_u)-1}}{3}.
\end{align*}
This result contains a small error which makes the result incorrect for $p=2$, $g_0=0$ (one really needs the $\epsilon$ in that case, see Proposition \ref{hulu2}). Secondly, in their proof they reduce to the case $n_u=0$, but they forget that if $n_u>0$, then more primes must ramify and hence the genus will grow faster.

Assume from now on that $n_u=0$. In fact Gold and Kisilevsky prove in an intermediate step
\begin{align*}
\liminf_{n \to \infty} \frac{g_n}{p^{2n}} \geq \frac{p-1}{2p^2}.
\end{align*}
Our result actually gives the tight bound
\begin{align*}
\liminf_{n \to \infty} \frac{g_n}{p^{2n}} \geq \frac{1}{2(p+1)}.
\end{align*}
Li and Zhao in \cite{LIC} construct a $\Z_p$-tower with the property
\begin{align*}
\lim_{n \to \infty} \frac{g_n}{p^{2n}} = \frac{1}{2(p+1)}.
\end{align*}
Li and Zhao furthermore write ``It would be interesting to determine if the bound of Gold and Kisilevsky is
the best and find some $\Z_p$-extension which realizes it.'' Our results show that their tower actually attains our limit.
\end{remark}

\subsection{Genus stability}

We will now introduce a special class of $\Z_p$-extensions of $K$. We are interested in 
classifying the cases when $g_n$ for large enough $n$ stabilizes. Note that $g_n$ is bounded below by a quadratic polynomial in $p^n$ by Corollary \ref{15b}.  We will now study the case where $g_n$ at some point becomes a quadratic polynomial in $p^n$. 
A $\Z_p$-tower $K_{\infty}/K$ is called geometric if $n_c=0$. 

\begin{proposition} \label{blaat}
 Let $K_{\infty}/K$ be a geometric $\Z_p$-extension. Then the following are equivalent:
\begin{enumerate}
\item There are $a,b,c \in \Q$, $m \in \Z_{\geq 0}$ such that for $n \geq m$ one has
\begin{align*}
g_n=a p^{2n}+b p^n+c.
\end{align*}
\item The extension $K_{\infty}/K$ is ramified at only finitely many places and for all $\pa$ the set 
\begin{align*}
\{i p^{-v(c_{\pa,i})}: (i,p)=1 \}
\end{align*}
has a maximum. 
\item The extension $K_{\infty}/K$ is ramified at only finitely many places  and for each $\pa$ which ramifies there are $a_{\pa} \in \Q_{>0}$ and $n_{\pa} \in \Z_{\geq 0}$ such that for $n \geq n_{\pa}$ one has
\begin{align*}
u_{\pa,n}=1+ a_{\pa} p^{n}.
\end{align*}

\end{enumerate}
\end{proposition}
\begin{proof}
i $\iff$ iii: This follows easily from Theorem \ref{831}. \\
ii $\iff$ iii: This follows directly from the definition of the $u_{\pa,n}$. 
\end{proof}

\begin{definition}
A geometric $\Z_p$-extension $K_{\infty}/K$ is called \emph{genus-stable} if one of the equivalent conditions of Proposition \ref{blaat} is satisfied.
\end{definition}

\begin{remark}
Let $L$ be a finite extension of $\Q_p$ with prime $\pa$ and ramification index $e=e(L/\Q_p)$. Let $L_{\infty}/L$ be a $\Z_p$-extension. In that case one has the following stability result  for the discriminants (which can be seen as the analogue of the genus).
There are  $A,B \in \Q$ such that $\disc(L_n/L)=\pa^{r_n}$ with $r_n=e (np^n)+A p^n + B$ for $n$ large enough. See \cite[Proposition 5.1]{TATE} for a proof. The reason that such a simple formula always holds is that $U_1$ is a finitely generated $\Z_p$-module in this case.
\end{remark}

\begin{remark} \label{917}
The definition of genus stability might look a bit arbitrary. However, it turns out that one can prove interesting results about genus stable towers. Here is an example. The $L$-functions of  genus stable covers of the projective line behave nicely in a $p$-adic way. One can show that the $p$-adic valuations of the inverses of the zeros of such $L$-function are uniformly distributed and form a finite union of arithmetic progressions in many cases. The latter result can only hold for genus stable covers. See \cite{WAN7} and \cite{KO9} for details. 

For future reference, let us discuss the degree of such $L$-functions. Let $K_{\infty}/K$ be a geometric $\Z_p$-tower. Let $\chi: \Gal(K_{\infty}/K) \to \C_p^*$ be a non-trivial finite character of order $p^{m_{\chi}}>1$. This character will factor through $\Gal(K_{m_{\chi}}/K)$. Then one can associate to this character an $L$-function
\begin{align*}
L(\chi,s)= \prod_{\pa} \frac{1}{1- \chi(\mathrm{Frob}(\pa)) s^{\deg(\pa)}} \in 1 + s \C_p [[s]].
\end{align*}
where the product is take over all primes of $K$ which are unramified in the extension $K_{m_{\chi}}/K$. Here $\mathrm{Frob}(\pa) \in \Gal(K_{m_{\chi}}/K)$ is the Frobenius element of $\pa$. By \cite[Theorem A]{BOM} $L(\chi,s)$ is a polygnomial of degree
\begin{align*}
\deg(L(\chi,s)) =& 2g(K)-2 + \mathrm{deg}(\mathfrak{f}_{m_{\chi}}) = 2 g(K)-2 + \sum_{\pa} \deg(\pa) u_{\pa,m_{\chi}},
\end{align*}
where $u_{\pa,m_{\chi}}$ are as before. Assume now that $K_{\infty}/K$ is genus stable and only ramified at rational primes. Let $\pa_1,\ldots, \pa_r$ be the ramifying primes in $K_{\infty}/K$ and set 
\begin{align*}
d_j p^{-m_j}=\max\{i p^{-v(c_{\pa_j,i})}: (i,p)=1 \}.
\end{align*}
where $d_j, m_j \in \Z_{\geq 0}$ and with $p \nmid d_j$.
Let $m=\max \{ m_j: j=1,\ldots,r \} $. Then if $m_{\chi}>m$, one has
\begin{align*}
\deg(L(\chi,s))  = 2 g(K)-2 + r + \sum_{j=1}^r d_j p^{m_{\chi}-m_j-1}. 
\end{align*}
Hence the degree of $L(\chi,s)$ is a linear polynomial in $p^{m_{\chi}}$ for large enough $m_{\chi}$. 
Conversely, if the degree of $L(\chi,s)$ is a linear polynomial in $p^{m_{\chi}}$ for large enough $m_{\chi}$, then the tower is genus stable.

\end{remark}

\subsection{Example: the projective line}

Let $K=k(X)$ be the function field of the projective line where $k$ is a finite field. We will study $\Z_p$-towers over $K$ which ramify only at rational points. 
For $x \in k$,  we set $\pi_x=X-x \in K$ and we set $\pi_{\infty}=1/X \in K$. Let $\alpha \in k$ with $\mathrm{Tr}_{k/\F_p}(\alpha) \neq 0$. Set $\beta=[\alpha]$. Analagous to Example \ref{183},  one can prove the following. Let $a=\wp y \in W(K)$ which gives rise to the extention $K(y)$  of $K$ which ramifies only at rational points (see \cite{WAN8} for the slightly more general case). 

\begin{lemma} \label{hulu}
The element $a$ is equivalent modulo $\wp W(K)$ to a unique element of the form
\begin{align*}
c \beta + \sum_{x \in k \cup \{\infty\}} \sum_{(i,p)=1} c_{x,i} [\pi_x]^{-i} \in W(K)
\end{align*}
with $c \in \Z_p$, $c_{x,i} \in \Z_q$ such that $c_{x,i} \to 0$ as $i \to \infty$. 
\end{lemma}
\begin{proof}
See \cite[Lemma 5.8]{WAN8}. This follows from Example \ref{183}. 
\end{proof}

Note that $a,a' \in W(K)$ give the same tower if and only if $a=a'c$ with $c \in \Z_p^*$.  One can easily see when this happens in Lemma \ref{hulu}.  We will now deduce data of the extension given by $a$. 

\begin{proposition} \label{hulu2}
Let $a= c \beta + \sum_{x \in k \cup \{\infty\}} \sum_{(i,p)=1} c_{x,i} [\pi_x]^{-i}= \wp (y_0,y_1,\ldots) \in W(K)$ as in Lemma \ref{hulu}.  Assume 
\begin{align*}
\mathrm{min}( \{v(c_{x,i}): x \in k \cup \{\infty\}, (i,p)=1\} \cup \{v(c)\})=0.
\end{align*}
Consider the tower $K_{\infty}/K$ corresponding to $a$ with subfield $K_n=K(y_0,y_1,\ldots,y_{n-1})$ of genus $g_n$. For $x \in k \cup \{\infty\}$,  set 
\begin{align*}
u_{x,n} = \left\{  \begin{array}{cc} 1+ \max\{ ip^{n-v(c_{x,i})-1}  : (i,p)=1 \textrm{ s.t. } v(c_{x,i})<n\} & \mathrm{if\ }\exists i: v(c_{x,i})<n \\	
0 & \mathrm{otherwise}. \end{array} \right.
\end{align*}
Then the extension $K_n/K$ is Galois with group isomorphic to $\Z/p^n\Z$. One has
\begin{align*}
n_u=n_c=\mathrm{min}\{  v(c_{x,i}): x \in k \cup \{\infty\},\ (i,p)=1\}
\end{align*}
and
\begin{align*}
\mathfrak{f}_n = \prod_{x \in k \cup \{\infty\}} (\pi_x)^{u_{x,n}}
\end{align*}
and 
\begin{align*}
p^{\min\{ n_c,n \}}(2 g_n - 2) = -2p^{n} + \sum_{x \in k \cup \{\infty\} } \sum_{j=0}^n \varphi(p^j) u_{x,j}.
\end{align*}
\end{proposition}
\begin{proof}
The results follows from the discussions before, and most importantly, Theorem \ref{831}.
\end{proof}

In the above proposition, one can easily deduce when the tower is genus stable with the help of Theorem \ref{blaat}. 

\begin{example}
Consider the unit root $\Z_p$-extension (called the Artin-Schreier-Witt extension in \cite{WAN7}) given by the unit root coefficient 
polynomial 
\begin{align*}
x=\sum_{(i,p)=1}^d [b_i X^{i}]=\sum_{(i,p)=1}^d [b_i][X^{i}] \in W(K)
\end{align*}
with $b_i \in k$ and $b_d \neq 0$, $d>0$ not divisible by $p$.  By the above equation, this defines a $\Z_p$-extension which is totally ramified at $\infty$. One finds for $n \geq 1$:
\begin{align*}
u_{\infty,n} = 1+ d p^{n-1}.
\end{align*}
and this gives
\begin{align*}
2g_n-2 = \frac{d}{p+1}p^{2n} - p^n - \frac{p+1+d}{p+1}.
\end{align*}
This is an example of a genus-stable tower.
\end{example}

\begin{remark}
Let $a=(a_0,a_1, \ldots) \in W(K)$ with $a_0 \not \in \wp K$. Consider the corresponding $\Z_p$-extension given by $a$. Let $\pa$ be a prime of  $K(X)$ of degree $d'$ over $\F_p$ which does not ramify in the tower. We give a geometric way to compute the Frobenius element $(\pa, K_{\infty}/K)$. Let $z \in \Ps^1(\overline{k})$ be a representative of $\pa$. Assume that $z$ is not a pole of the $a_i$ (otherwise, we have to find another representative of $a \pmod{\wp W(K)}$; or one can assume $a$ is in our unique reduced form). Set $a(z)=(a_0(z),a_1(z),\ldots) \in W(k(z))$. Let $y \in \wp^{-1} z \in W(\overline{k})$.  
One has
\begin{align*}
F^{d'}y = y + \sum_{j=0}^{d'-1} F^j(Fy-y) = y+ \mathrm{Tr}_{W(k(z))/W(\F_p)} ( a(z) ).
\end{align*}
This shows that the Frobenius is equal to 
\begin{align*}
(\pa, K_{\infty}/K)= \left(\overline{a} \mapsto -\mathrm{Tr}_{W(k(z))/\Z_p} (a(z)) \right) \subseteq \Hom(\Z_p \overline{a}, \Z_p)  \cong \Gal(K(\wp^{-1}a)/K).
\end{align*} 
A similar formula works for primes which are not ramified in say $K_n/K$. Furthermore, this formula generalizes when $K$ is replaced by another function field.
\end{remark}

\section{Thanks}
We would like to thank Chris Davis for his help with Witt-vectors and for his proofreading of parts of this manuscript.

\end{document}